\documentclass{amsart}

\usepackage{amsmath,amssymb,amsthm,enumerate}
\usepackage[colorlinks,citecolor=red,pagebackref,hypertexnames=false]{hyperref}

\usepackage{graphicx,tikz, caption}
\usetikzlibrary{calc}

\newtheorem{theorem}{Theorem}
\newtheorem*{thm}{Theorem}

\newtheorem{corollary}[theorem]{Corollary}

\newtheorem{lemma}[theorem]{Lemma}
\newtheorem{example}[theorem]{Example}

\newtheorem*{definition}{Definition}

\theoremstyle{remark}

\begin{document}

\title[]{Single radius spherical  cap discrepancy \\via gegenbadly approximable numbers}

\author[]{Dmitriy Bilyk \and Michelle Mastrianni \and Stefan Steinerberger}
\address{School of Mathematics, University of Minnesota, Minneapolis, MN 55455, USA}
\email{dbilyk@umn.edu}

\address{School of Mathematics, University of Minnesota, Minneapolis, MN 55455, USA}
\email{michmast@umn.edu}

\address{Department of Mathematics, University of Washington, Seattle, WA 98195, USA}
\email{steinerb@uw.edu}

\begin{abstract} A celebrated result of Beck shows that for any set of $N$ points on $\mathbb{S}^d$ there always exists a spherical cap $B \subset \mathbb{S}^d$ such that number of points in the cap deviates from the expected value $\sigma(B) \cdot N$ by at least $N^{1/2 - 1/2d}$, where $\sigma$ is the normalized surface measure. We refine the result and show that, when $d \not\equiv 1 ~(\mbox{mod}~4)$, there exists a (small and very specific) set of real numbers such that for every $r>0$ from the set one is always guaranteed to find a spherical cap $C_r$ with the given radius $r$ for which the result holds. The main new ingredient is a generalization of the notion of badly approximable numbers to the setting of Gegenbauer polynomials: these are fixed numbers $ x \in (-1,1)$ such that the sequence of Gegenbauer polynomials $(C_n^{\lambda}(x))_{n=1}^{\infty}$ avoids being close to 0 in a precise quantitative sense.
\end{abstract}

\maketitle

\section{Introduction and Results}

\subsection{Gegenbauer Polynomials}\label{s.GP}  The Gegenbauer polynomials $C_n^{\lambda}:\mathbb{R} \rightarrow \mathbb{R}$ are defined recursively via
$C_0^{\lambda}(x)= 1$, $C_1^{\lambda}(x) = 2\lambda x$ and
\begin{align*}
C_n^{\lambda}(x) = \frac{n+\lambda -1}{n}  2x \cdot C_{n-1}^{\lambda}(x) - \frac{n + 2\lambda - 2}{n} \cdot C_{n-2}^{\lambda}(x).
\end{align*} 
Gegenbauer polynomials are orthogonal on the interval $[-1,1]$ with respect to the weight $w(x) = (1-x^2)^{\lambda-1/2}$. They also satisfy the normalization
$$ \int_{-1}^{1} C_n^{\lambda}(x)^2 (1-x^2)^{\lambda-1/2} dx = \frac{\pi 2^{1-2\lambda}}{\Gamma(\lambda)^2} \frac{\Gamma(n+2\lambda)}{n! (n+\lambda)} = c_{\lambda} \cdot n^{2\lambda - 2} + \mathcal{O}(n^{2\lambda -1}).$$
In particular, we see that for a typical $x \in (-1,1)$ we would expect that, at least on average, $|C_n^{\lambda}(x)| \sim n^{\lambda - 1}$. We note that 
larger growth is possible at the boundary since the weight $w(x) $ decays at $\pm 1$. In particular, $C_n^\lambda (1) 
 \sim n^{2\lambda -1}$.
Our main question is whether there exist suitable $x \in (-1,1)$ such that $|C_n^{\lambda}(x)|$ is always `large'. The
difficulty comes from the oscillatory nature of these polynomials, one needs to ensure that $x$ is always `far' away from a root, where
the notion of `far' shrinks with $n$. We will now formally define these numbers.

\begin{definition}\label{d.ggb} Let $\lambda > 0$. We say that $x \in (-1,1)$ is \emph{gegenbadly approximable} (with respect to $\lambda > 0$) if there exists a constant $c_x > 0$ such that, for all $n \in \mathbb{N}$,
\begin{equation}\label{e.ggb} |C_n^{\lambda}(x)| \geq c_x \cdot n^{\lambda - 1}.
\end{equation}
\end{definition}

We start by explaining our choice of terminology. 
Recall that a real number $\alpha \in \mathbb{R}$ is said to be \textit{badly approximable} if there exists $c_{\alpha} > 0$ such that 
$$ \forall~\frac{p}{q} \in \mathbb{Q} \quad  \qquad \left| \alpha - \frac{p}{q} \right| \geq \frac{c_{\alpha}}{q^2}.$$
The case $\lambda = 1$ essentially reduces to
the classical notion of badly approximable numbers. This is made precise in the next Theorem which also shows that the case where
$\lambda = 2k + 1 \in \mathbb{N}$ is an odd integer is somewhat atypical.

\begin{theorem} \label{thm:1}  If $\lambda$ is an odd integer, no $\lambda-$gegenbadly approximable numbers exist. Moreover, for every $x\in (-1,1)$, there exists a constant $c_x \ge 0$, such that
$$\left| C_n^{\lambda}(x)\right| \le c_x \cdot  n^{\lambda-2} \quad \mbox{holds for infinitely many values} \quad n \in \mathbb N.$$
When $\lambda = 1$ and  $x  \in (-1,1)$, then there exists  $c_x > 0$ such that for all $n\in \mathbb N$
$$  |C_n^{(1)}(x)|  >  \frac{c_x}{n^{}}$$
\emph{if and only if} $x = \cos{\theta}$ where $\theta/\pi$ is a badly approximable number.
\end{theorem}

\begin{center}
\begin{figure}[h!]
\begin{tikzpicture}
\node at (0,0) {\includegraphics[width=0.45\textwidth]{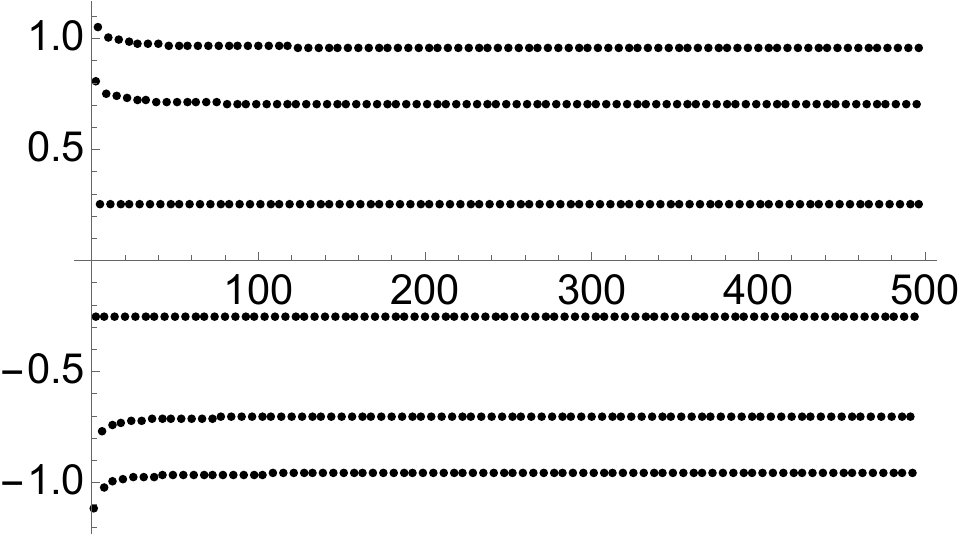}};
\node at (6,0) {\includegraphics[width=0.45\textwidth]{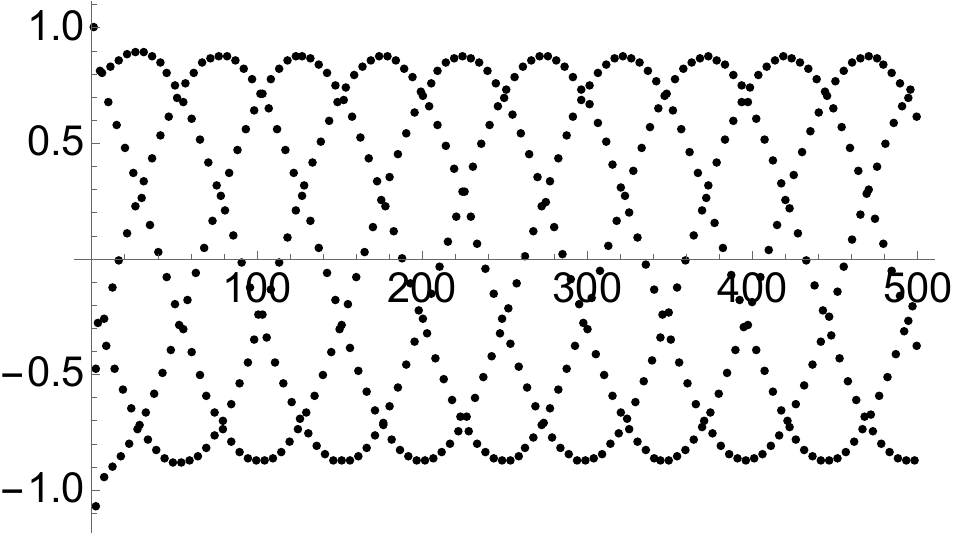}};
\end{tikzpicture}
\caption{Left: the sequence $n^{-1/2} \cdot C_n^{3/2}(1/2)$ for $1\leq n \leq 500$. Right: the sequence $ n^{-1/2} \cdot  C_n^{3/2}(1/3)$ for $1 \leq n \leq 500$. The number $x=1/2$ is $3/2-$gegenbadly approximable, $x=1/3$ is not.}
\end{figure}
\end{center}

\vspace{-10pt}
We will always restrict ourselves to the case where $\lambda > 0$ is not an odd integer.
Our next theorem suggests that gegenbadly approximable numbers do exist for  other values of $\lambda >0$ and  characterizes the set of such numbers. This result also provides a way of explicitly constructing gegenbadly approximable numbers.
\begin{theorem} \label{thm:2} Suppose $0 < \lambda < \infty$ is \emph{not} an odd integer.  If $x \in (-1,1)$ is $\lambda-$\emph{gegenbadly approximable}, then
\begin{enumerate}[(i)]
\item\label{cond1} $x = \cos{(\pi (p/q))}$ where $p/q \in \mathbb{Q}$
\item\label{cond2}  for all $1 \leq n \leq q$
$$ \lambda \left( 2p -q \right) + 2 np -q \qquad \mbox{is not divisible by}~2q.$$
\end{enumerate}
Conversely, for any $x$ satisfying these two properties, there exists $N_{x.\lambda} \in \mathbb{N}$ such that either there exists $n \leq N_{x,\lambda}$ such that $C_n^{\lambda}(x) = 0$ or $x$ is $\lambda-$gegenbadly approximable. Moreover, we have
\begin{equation}\label{e.Nx}
 N_{x,\lambda} \leq  \frac{\pi}{2}   \frac{\lambda |1-\lambda| \cdot \cot(\pi \frac{p}{q})}{ \min_{1 \leq n \leq q} \left\|  (n+\lambda) \frac{p}{q} - \frac{\lambda+1}{2} \right\| },
 \end{equation}
where $\|x\| = \min_{m \in \mathbb{Z}} |x - m|$ is the distance to the nearest integer.
\end{theorem}
\textit{Remarks:}
\begin{enumerate}
\item If conditions \eqref{cond1} and \eqref{cond2} are satisfied,  our argument implies the inequality $|C_n^{\lambda}(x)| \geq c_x \cdot  n^{\lambda-1}$ for all $n$ sufficiently large. 
  $N_{x,\lambda}$ in \eqref{e.Nx} quantifies `sufficiently large' to ensure that it holds for all $n$.
\item The result makes it fairly easy to check whether any given number is $\lambda-$gegenbadly approximable with a finite amount of computation.
\item For most real $\lambda > 0$ condition \eqref{cond2} is automatically satisfied since then the number $\lambda \left( 2p -q \right) + 2 np -q$ will usually not even be an integer.
\item The cases that carry the most geometric significance are when the  Gegenbauer polynomials $C_n^{\lambda}$ correspond to a natural rotationally invariant function basis on the sphere $\mathbb S^d$, which happens exactly when $\lambda = (d-1)/2$, i.e. $2\lambda \in \mathbb{N}$ and in these cases, condition \eqref{cond2} is often relevant.

\item Checking the first few values of $n$ can usually not be avoided. For example, it is not difficult to see that the function $\lambda \rightarrow C_{11}^{\lambda}(\cos{(\pi/5)})$ has a sign change around $\lambda \sim 7.918\dots$. The conditions \eqref{cond1} and \eqref{cond2} are satisfied for $\lambda$ close to that number for $x=\cos{(\pi/5)}$ while there exists a $\lambda^*  \sim 7.918$ for which $C_{11}^{\lambda^*}(\cos{(\pi/5)}) = 0$. Thus, while $|C_n^{\lambda^*}(\pi/5)| \geq c \cdot n^{\lambda^*-1}$  is asymptotically true, it fails for $n=11$. We have $N_{\cos{(\pi/5)}, \lambda^*} \sim 1570 \geq 11$. Many other such examples can be constructed. 
\end{enumerate}

We now give a particular example of gegenbadly approximable numbers which will be relevant to applications (see Corollary \ref{cor:6}). 

\begin{example} \label{ex:gegen}
The numbers $\pm 1/2$ and $\pm 1/\sqrt{2}$ are $3/2-$gegenbadly approximable.
\end{example}
This follows from $1/2 = \cos{(\pi/3)}$ and $1/\sqrt{2}  = \cos{(\pi/4)}$. The numbers $N_{x,\lambda}$ are easy to compute in these cases (8 and 9, respectively) which then implies the desired conclusion after checking that the first $N_{x,\lambda}$ terms are not 0.  These two examples are actually quite special: there are only two gegenbadly approximable numbers that are rational. There are only six such numbers whose square is rational.

\begin{corollary}\label{c.quadset}
Assume that  $x \in (-1,1)$ satisfies $x^2 \in \mathbb{Q}$. Then $x$  is  $\lambda-$gegenbadly approximable for some $\lambda > 0$ if and only if 
\begin{equation}\label{e.quadset}
 x \in \left\{ \pm \frac{1}{2}, \pm \frac{1}{\sqrt{2}}, \pm \frac{\sqrt{3}}{2} \right\} .
 \end{equation}
\end{corollary}
The proof of Corollary \ref{c.quadset} is presented in Section \ref{sec.proofcol}. We  remark that the characterization of gegenbadly approximable numbers akin to Theorem \ref{thm:1} could be completed if one could describe the set of values $x = \cos \theta \in (-1,1)$, with $\theta$ a rational multiple of $\pi$,  which are not roots of Gegenbauer polynomials for any $n\in \mathbb N$ for a fixed $\lambda >0$. We are not aware of any relevant results in this direction. It is proved in \cite{render} that for many values of $\lambda>0$, one has $C_n^\lambda (\sqrt{b/c}) \neq 0$  when $b,c \in \mathbb N$ are relatively prime and $b\not\in \{ 1,3\}$. Corollary \ref{c.quadset} shows that this result does not apply to our setting. It may be most natural to restrict oneself to the case $2\lambda \in \mathbb{N}$.

\subsection{Single Radius Spherical  Cap Discrepancy}
We now turn to an interesting application of the results about gegenbadly approximable numbers to discrepancy theory and irregularities of distribution. 
 For  a point $x \in \mathbb{S}^d$ and $t \in [-1,1]$,  the spherical cap $B(x,t) \subset \mathbb{S}^d$  of `height' $t$ centered at $x$ is defined  by
\begin{equation}\label{e.cap}
 B(x,t) = \left\{ y \in \mathbb{S}^d: x \cdot y \geq t \right\}.
 \end{equation}
Given a finite set of points $Z = \left\{z_1, \dots, z_N\right\} \subset \mathbb{S}^d$, one could  argue that this set  is regularly distributed if, with respect to all spherical caps $B(x,t)$, the proportion of points that lies in the spherical cap is close to the volume of the spherical cap, i.e.
$$ \frac1{N} \cdot  \# \left( Z \cap B(x,t)\right) \sim \sigma(B(x,t)),$$
where $\sigma$ is the normalized surface measure on $\mathbb{S}^d$. This motivates the definition of the $L^2$ spherical cap discrepancy as
$$ D_{L^2, {\small \mbox{cap}}}(Z)^2 = \int_{-1}^1 \int_{\mathbb{S}^d} \left| \frac{ \# \left( Z \cap B(x,t)\right) }{N} - \sigma(B(x,t)) \right|^2 d\sigma(x) dt.$$ 
This is the average of the squared difference of the aforementioned quantities over all spherical caps {\it{of all sizes}}.  There is a different way of arriving at this quantity (motivated by sums of pairs of distances), known as the \textit{Stolarsky Invariance Principle} \cite{stol1, stol2} which will not be relevant for our paper (but see \cite{bil2,brauch}). \\

We start by introducing a seminal result of Beck \cite{beck} showing that the discrepancy  of any discrete point set can never be too small: avoiding irregularity is impossible.
\begin{thm}[J. Beck, \cite{beck}] For any set  $Z = \left\{z_1, \dots, z_N\right\} \subset \mathbb{S}^d$, we have
\begin{equation}\label{e.beck}
 \int_{-1}^1 \int_{\mathbb{S}^d} \left| \frac{ \# \left( Z \cap B(x,t)\right) }{N} - \sigma(B(x,t)) \right|^2 d\sigma(x) dt \geq c_d \cdot N^{-1- \frac{1}{d}}.\end{equation}
\end{thm}

This result is known to be best possible and has inspired a lot of subsequent work (see 
\cite{aist, bilyk, bil3, et, fer, wag1, wag2}).
It implies, for example, that for any set of $N$ points on $\mathbb{S}^2$ there always exists a spherical cap containing $\sim  N^{1/4}$ more or
less points than expected.  This statement is also asymptotically optimal, up to logarithms.  \\

Averaging over all sizes (radii) was crucial in the proof of Beck's theorem. The phenomenon of  averaging over size (radii, side length, dilations) also appeared in the proofs of discrepancy bounds similar to \eqref{e.beck} in other geometric contexts: rotated rectangles \cite{beck-rect}, disks (balls) in the torus \cite{Montgomery}, convex sets \cite{beck,BT}, and many others. The  example of disks  in the torus \cite{Montgomery}  is particularly interesting as it allowed for averaging over only two values: it implied that there is a disk of radius $1/4$ or $1/2$ with large discrepancy, but it is still absolutely unclear whether only one of these radii is enough (see \cite{beck88} for some results in this direction). We also refer to a recent more general result of Brandolini-Gariboldi-Gigante \cite{bran} that is in a similar spirit. More generally, in all of these geometric settings it remained an open question whether averaging over sizes is necessary, or if large discrepancy is already achieved by  a single size? \\

Our main result is a refinement of  Beck's theorem, which answers the question above in the case of the spherical cap discrepancy. It is enough to consider spherical
caps of a \textit{fixed} size provided that size is gegenbadly approximable.

\begin{theorem}\label{t.main} Let  $d \not \equiv 1 ~\mbox{mod}~4$.
Suppose $t \in (-1,1)$ is $(d+1)/2-$gegenbadly approximable. Then for some $c_{d,t} >0$ and all $Z = \left\{z_1, \dots, z_N\right\} \subset \mathbb{S}^d$ we have
$$  \int_{\mathbb{S}^d} \left| \frac{ \# \left( Z \cap B(x,t)\right) }{N} - \sigma(B(x,t)) \right|^2 d\sigma(x)  \geq c_{d,t} \cdot N^{-1 - \frac{1}{d}}.$$
\end{theorem}
Theorem \ref{t.main} can be considered a de-randomization of Beck's result since one layer of averaging is removed. 
The following corollary follows from the above theorem and the fact that $t = 1/2$ and $t = 1/\sqrt{2}$ are $3/2$-gegenbadly approximable (Example \ref{ex:gegen}).
\begin{corollary} \label{cor:6}
For $t \in \{1/2, 1/\sqrt{2}\}$ and any set of $N$ points on $\mathbb{S}^2$, there always exists $x \in \mathbb{S}^2$ such that $$\left| \frac{ \# \left( Z \cap B(x,t)\right) }{N} - \sigma(B(x,t)) \right| \ge c \cdot N^{1/4}.$$
\end{corollary}

We conclude with some related remarks.

\subsubsection{The condition $d \not \equiv 1 ~\mbox{mod}~4$}  This condition, which is equivalent to the condition that $\lambda$ is not an odd integer in Theorem \ref{thm:2},  is necessary for the existence of gegenbadly approximable numbers in view of Theorem \ref{thm:1}. It is also relatively easy to see that the statement of Theorem \ref{t.main} is false when $d=1$. The first open case that is not covered by our results is the case of $\mathbb{S}^5$. We notice that the same restriction $d \not \equiv 1 ~\mbox{mod}~4$ appeared in several results in analysis and discrepancy theory related to balls in $\mathbb R^d$: lattice points in balls \cite{KSS}, orthogonal exponentials on balls \cite{IR}, $L^2$ discrepancy with respect to balls \cite{CT}. In all of these results, this effect appeared due to a phase shift in the oscillatory term of  the asymptotic relations for the Bessel functions. In our case, it stems from a similar phase shift in the asymptotics of the Gegenbauer polynomials, see \eqref{e.KTz} and \eqref{e.NG}. 

\subsubsection{The `freak theorem' and uniform distribution} These single radius spherical cap discrepancy results are related to other interesting problems on the sphere. Consider the following two questions: assume a function $f \in C (\mathbb S^d)$ has mean zero on every spherical cap $B(x,t)$ of fixed size $t$, is it true that $f=0$? Assume that a sequence of points on the sphere $\mathbb S^d$ is uniformly distributed  with respect to every spherical cap $B(x,t)$ of fixed size $t$, is it true that it is a  uniformly distributed sequence? The answer to both questions turns out to be \textit{no}! The former has been answered by Ungar (\cite{ungar}, $d=2$) and Schneider (\cite{schneider}, $d\ge 3$) and is sometimes referred to as the `freak theorem', while the latter was dealt with by Vol\v{c}i\v{c} \cite{volcic}.  The exceptional values of $t\in (-1,1)$ in both cases are exactly the roots of Gegenbauer  polynomials $C_n^\lambda$ corresponding to dimension $d+2$, i.e. with $\lambda =(d+1)/2$ as in Theorem \ref{t.main}. Our result lies on the opposite end of the spectrum, as being a root  may be viewed as an extreme case of failing the gegenbadly approximable property. 

\subsubsection{Characterizing single radius spherical cap discrepancy} Theorem \ref{thm:2} characterizes gegenbadly approximable numbers: it characterizes $x \in (-1,1)$ for which the inequality $|C_n^{\lambda}(x)| \geq c_x \cdot  n^{\lambda-1}$ is true for all $n$ sufficiently large and it also quantifies `sufficiently large' turning it into an effective computable criterion. 
Theorem \ref{t.main} shows that gegenbadly approximable numbers admit an irregularities of distribution phenomenon with respect to a  spherical caps of fixed radius (provided that the corresponding height is $(d+1)/2-$gegenbadly approximable). The main purpose of Theorem \ref{t.main} is to show that such radii indeed exist, however, it is unlikely that it provides a characterization of such radii. This raises a natural question which, on $\mathbb{S}^2$, assumes the following form: Corollary \ref{cor:6} is true for $t \in \{1/2, 1/\sqrt{2}\}$, for which other values of $t$ does it hold? It seems likely that the condition $|C_n^{\lambda}(x)| \geq c_x \cdot  n^{\lambda-1}$ is only required to be true on average (in some suitable sense) and it is conceivable that Corollary \ref{cor:6} is true for many more values of $t$. We believe this to be an interesting question which is going to require a completely new idea.
The converse question, which can be viewed as a quantitative version of the aforementioned result of Vol\v{c}i\v{c} \cite{volcic}, is also interesting, but wide open: under which conditions on the radius there always exist discrete sets of $N$ points on the sphere such that their single radius spherical cap discrepancy is `small'?

\section{Proof of Theorem \ref{thm:1}}

 We use a result of Kalton \& Tzafriri \cite{kalton}: for any $0 < \lambda < \infty$, there exists a constant $c_{\lambda}$ such that, provided $n$ is sufficiently large, meaning $n \sin{\theta} \geq 1$,
\begin{equation}\label{e.KTz}
 \left| C_n^{\lambda}(\cos{\theta}) - \frac{n^{\lambda-1}}{2^{\lambda-1} \Gamma(\lambda)} \frac{1}{(\sin{\theta})^{\lambda}} \cos \left((n+\lambda)\theta - \lambda \frac{\pi}{2} \right) \right| \leq \frac{c_{\lambda}}{(\sin{\theta})^{\lambda +1}} n^{\lambda-2}.
 \end{equation}
This asymptotic expansion shows that it suffices to understand, for $n \geq (\sin{\theta})^{-1}$, whether there exists a constant $c>0$ (depending on $\lambda, \theta$) such that, uniformly for all $n$ sufficiently large
$$ \left| \cos \left((n+\lambda)\theta - \lambda \frac{\pi}{2} \right) \right| \geq c_{\theta, \lambda} > 0.$$
The cosine vanishes, $\cos {x} = 0$, for reals of the form $x  = (m+1/2) \pi$ where $m \in \mathbb{Z}$. Thus, by continuity and periodicity, our problem is equivalent to understanding whether there exists another constant $c^*_{\theta, \lambda}$ such that for all sufficiently large $n$,
$$ \inf_{m \in \mathbb{Z}} \left| (n+\lambda)\theta - \lambda \frac{\pi}{2} - \left(m + \frac{1}{2} \right)\pi \right| \geq c^*_{\theta, \lambda} > 0,$$
or equivalently,  whether there exists another constant $c^{**}_{\theta, \lambda} = c^*_{\theta, \lambda}/\pi  > 0$ such that for all sufficiently large $n$
$$ \inf_{m \in \mathbb{Z}} \left| (n+\lambda)\frac{\theta}{\pi} - \left(m + \frac{\lambda + 1}{2} \right) \right| \geq c^{**}_{\theta, \lambda} > 0.$$
It is convenient to denote the distance  from $x$ to the nearest integer  by $\| x \|$, i.e.
$$ \left\|x \right\| = \min_{k \in \mathbb{Z}} |x - k|.$$
Using this abbreviation, we are interested in studying
$$ \inf_{m \in \mathbb{Z}} \left| (n+\lambda)\frac{\theta}{\pi} - \left(m + \frac{\lambda + 1}{2} \right) \right| = \left\| (n+\lambda)\frac{\theta}{\pi} -  \frac{\lambda + 1}{2}  \right\| $$
Now we make a case distinction.

\subsection{$\lambda$ is an odd integer.}
If $\lambda$ is an odd integer, then $(\lambda+1)/2 \in \mathbb{N}$ and the question simplifies to whether there exists $ c_{\theta, \lambda}^{***} > 0$ such that for all $n \in \mathbb{N}$
$$ \inf_{n \in \mathbb{N}} \left\| (n+\lambda) \frac{\theta}{\pi}  \right\|  \geq c_{\theta, \lambda}^{***} > 0.$$
If $\theta/\pi = p/q \in \mathbb{Q}$ is a rational number, then it is clear that no such constant exists since the expression is 0 whenever $n + \lambda$ is a multiple of $q$. We can thus assume that $\theta/\pi$ is irrational. In that case, however, the sequence $(n+ \lambda)\theta/\pi~\mod 1$ is an irrational rotation on the torus and thus uniformly distributed and, in particular, gets arbitrarily close to 0 (meaning an integer). Indeed, we can make this more quantitative: by Dirichlet's pigeonhole principle we can conclude that
$$ \min_{1 \leq n \leq N} \left\| (n+\lambda) \frac{\theta}{\pi}  \right\| \lesssim \frac{1}{N}.$$
This implies that for infinitely many $n \in \mathbb{N}$
$$ \left| \cos \left((n+\lambda)\theta - \lambda \frac{\pi}{2} \right) \right| \lesssim \frac{1}{n}$$
from which we deduce that for infinitely many $n \in \mathbb{N}$
$$ \left| C_n^{\lambda}(\cos{\theta})\right| \lesssim n^{\lambda-2}.$$

\subsection{The case $\lambda = 1$}
 If $\lambda = 1$, then the previous argument already implies that
 $$ \left| C_n^{\lambda}(\cos{\theta})\right| \lesssim \frac{1}{n} \qquad \mbox{for infinitely many}~n.$$
We will now show that this result is sharp. When $\lambda=1$, the Gegenbauer polynomials $C^{1}_n$ simplify and
$ C_n^{1}(x) =  U_n(x)$,
where $U_n(x)$ denotes the Chebychev polynomials of the second kind.
Chebychev polynomials of the second kind satisfy
\begin{equation}\label{e.cheb}
U_n(\cos{(\theta)}) = \frac{\sin{((n+1)\theta)}}{\sin{(\theta)}}
\end{equation}
for $\theta \in (0,\pi)$.
We will now analyze the cases of $\theta/\pi$ being irrational and $\theta/\pi$ being rational. If $\theta/\pi= p/q$
is rational, then
$ \sin{\left( (n+1) \theta\right)} =  \sin{\left( \pi (n+1) p/q \right)}$
 cannot be bounded away from 0 and will assume the value $0$ infinitely many times at regularly spaced intervals. We will now argue that there exists a constant $c_{\theta} > 0$ such that
$$ \forall~n \in \mathbb{N}: \qquad \left| U_n(\cos{(\theta)})  \right| \geq \frac{c_{\theta}}{n}$$
if and only if $\theta/\pi$ is badly approximable.
Identity \eqref{e.cheb} 
shifts the question to whether $n\theta$ can ever be close to being a multiple of $\pi$.
This question is well studied in classical number theory and we see that we have
$$ \left\| \frac{n \theta}{\pi}  \right\| =  \inf_{m \in \mathbb{N}} \left| \frac{n \theta}{\pi} - m \right| \geq \frac{c_{\theta}}{n}$$
for some universal $c_{\theta} > 0$ if and only if $\theta/\pi$ is a badly approximable number.

\subsection{Proof of Corollary \ref{c.quadset}}\label{sec.proofcol}
\begin{proof}
If $x^2 = r \in \mathbb{Q}$ and $x$ is gegenbadly approximable, then, according to Theorem \ref{thm:1},  $\arccos{(\sqrt{r})}$ is a rational multiple of $\pi$. A result of Varona \cite{varona} (see also \cite{lehmer}) states this is  the case if and only if 
$ r \in \left\{ 0, 1/4, 1/2, 3/4, 1\right\}.$ If
$r=1$, then $x = \pm 1$ is not in $(-1,1)$.

 It is also easy to discard the case $x= 0 = \cos{(\pi/2)} = 0$, since in that case $p=1$ and $q=2$, and we see that 
 $$2n - 2 = \lambda(2p - q) + 2np - q = 2qm = 4m$$ 
 has a solution ($n=1$ and $m=4$), hence $x=0$ is not gegenbadly approximable. Thus \eqref{e.quadset} holds.

 Conversely, assume \eqref{e.quadset}. 
The cases $r=1/4$ and $r=1/2$, corresponding to $x= \pm 1/2$ and $x=\pm 1/\sqrt{2}$, have already been dealt with in Example \ref{ex:gegen}. It remains to deal with $x = \pm \sqrt{3}/2 = \pm \cos{(\pi /6)}$. The condition that remains to be checked is whether, for $p=1$ and $q=6$, the equation
$$ -4 \lambda + 2n - 6 = \lambda(2p - q) + 2np - q = 2qm = 12m $$
can avoid having a solution. Setting $\lambda = 4/3$, it is not hard to check that $\sqrt{3}/2$ is indeed $4/3-$gegenbadly approximable. This can be  be seen by computing $N_{\sqrt{3}/2, 4/3} \sim 21.7$ and checking the first 22 elements of the sequence.
\end{proof}

\section{Proof of Theorem \ref{thm:2}}
The proof comes in two parts: in the first part we show that any $\lambda-$gegenbadly approximable numbers has to necessarily have the desired form. In the second part, we will show that these properties are sufficient in the sense that $|C_n^{\lambda}(x)| \gtrsim n^{\lambda - 1}$ for all $n$ sufficiently large: the remaining problem is to quantify `sufficiently large'.

\subsection{Part 1: necessity.}
Let us now assume that $\lambda$ is not an odd integer and $x = \cos \theta$ is $\lambda-$gegenbadly approximable. Recalling the beginning of the proof of Theorem \ref{thm:1}, it remains to see whether there is a uniform bound
 $$ \inf_{n \in \mathbb{N}} \left\|  (n+\lambda) \frac{\theta}{\pi} - \frac{\lambda+1}{2} \right\| \geq  c_{\theta, \lambda}^{**}  > 0.$$
If $\theta/\pi$ is irrational, then $n \theta/\pi$ is dense modulo 1, and therefore so is $(n+\lambda) \theta/\pi$ and no such estimate can be true. We therefore deduce  that $\theta/\pi = p/q$ is rational, i.e. condition \eqref{cond1} of Theorem \ref{thm:2} holds. The problem thus simplifies to understanding whether there exists $c_{\theta, \lambda}^{**}  > 0$ such that 
$$ \inf_{n \in \mathbb{N}} \left\|  (n+\lambda) \frac{p}{q} - \frac{\lambda+1}{2} \right\| > c_{\theta, \lambda}^{**} > 0.$$
This is a $q-$periodic sequence and it suffices to understand whether
$$ \min_{1 \leq n \leq q} \left\|  (n+\lambda) \frac{p}{q} - \frac{\lambda+1}{2} \right\| > 0.$$
This minimum will be positive unless it vanishes somewhere which happens if there exists $1 \leq n \leq q$ and $m \in \mathbb{Z}$ such that
$$  (n+\lambda) \frac{p}{q} - \frac{\lambda+1}{2} = m$$
or, equivalently,  unless there exists $1 \leq n \leq q$ and $m \in \mathbb{Z}$ such that
$$ \lambda \left(\frac{p}{q} - \frac{1}{2}\right) + n \frac{p}{q} = m + \frac{1}{2}.$$
Multiplying both sides with $2q$ we see that this happens if and only if condition \eqref{cond2} in Theorem \ref{thm:2} is satisfied, i.e. for some $1 \leq n \leq q$ there exists an $m \in \mathbb{Z}$ such that
$$ \lambda \left( 2p -q \right) + 2 np = 2qm + q$$
which is equivalent to saying that
$$ \lambda \left( 2p -q \right) + 2 np -q \qquad \mbox{is divisible by}~2q~\mbox{for some}~1 \leq n \leq q.$$

\subsection{Part 2: sufficiency.}
If we now assume conditions \eqref{cond1}  and \eqref{cond2}  to hold, then we know that
$$ X = \inf_{1 \leq n \leq q} \left\|  (n+\lambda) \frac{p}{q} - \frac{\lambda+1}{2} \right\| > 0$$
since Condition \eqref{cond1} implies that $\theta/\pi$ is rational and condition \eqref{cond2} was derived to be equivalent to this statement.
Since $X$ denotes a distance to a nearest integer, we also have $0 \leq X \leq 1/2$.
 The function $\left|\cos{(x)}\right|$ is $\pi-$periodic and thus
\begin{align*}
 \left| \cos \left((n+\lambda)\theta - \lambda \frac{\pi}{2} \right)\right|  &=  \left|  \cos \left[ \pi \left( (n+\lambda)\frac{p}{q} -  \frac{\lambda + 1}{2} \right) + \frac{\pi}{2} \right] \right| \\
  &= \left| \cos \left[ \pi \left\| (n+\lambda)\frac{p}{q} -  \frac{\lambda + 1}{2} \right\| + \frac{\pi}{2} \right] \right| \\
  &= \sin\left( \pi \left\| (n+\lambda)\frac{p}{q} -  \frac{\lambda + 1}{2} \right\|  \right) \\
  &\geq \frac{2}{\pi} \left\| (n+\lambda)\frac{p}{q} -  \frac{\lambda + 1}{2} \right\|  \geq \frac{2}{\pi} X.
\end{align*}
We invoke an asymptotic expansion of Natanson--Glagovskii \cite{nat} who showed that
\begin{equation}\label{e.NG}
C_n^{\lambda}(\cos{\theta}) = \frac{A_n}{(\sin{\theta})^{\lambda}} \left( \cos{\left( (n+\lambda) \theta - \lambda \frac{\pi}{2} \right)} + \alpha_n(\theta) \frac{\lambda (1-\lambda) \cot \theta}{n+\lambda}  \right),
\end{equation}
where $|\alpha_n(\theta)| \leq 1$. It remains to check the validity of the inequality
$$  \frac{\lambda |1-\lambda| \cot \theta}{n+\lambda} < \frac{2}{\pi} X,$$ 
which is easily seen to be satisfied as soon as
$$ n \geq  \frac{\pi}{2}    \frac{\lambda |1-\lambda| \cot \theta }{X}.$$
This means that the Gegenbauer polynomial never vanishes and from the previous arguments we already know that it has the correct asymptotic growth as $n \rightarrow \infty$.
This establishes the desired result.

\section{Lower Bound for the Single Radius Spherical Cap Discrepancy: Proof  of Theorem \ref{t.main}}
\subsection{Preliminary Results} We first collect a number of relevant results. 
Throughout the remainder of this section, we will always set
$$ \lambda = \frac{d-1}{2}.$$
For any bounded function $f:[-1,1] \rightarrow \mathbb{R}$, we can define an associated notion of discrepancy for any set
$ Z = \left\{z_1, \dots, z_N \right\} \subset \mathbb{S}^d$ via
 via 
$$ \left[ D_{L^2, f}(Z) \right]^2 = \int_{\mathbb{S}^d} \left| \int_{\mathbb{S}^d} f(x \cdot y) d\sigma (y) - \frac{1}{N} \sum_{j=1}^{N} f(x \cdot z_j) \right|^2 d\sigma(x).$$
This notion can be defined for any function $f:[-1,1] \rightarrow \mathbb{R}$ but, naturally, some are easier to interpret than others. We will focus on the 
spherical cap discrepancy with fixed radius  which corresponds to the function
$$ f_t(\tau) = \mathbf{1}_{[t, 1]}(\tau),$$
where the parameter $t\in (-1,1)$ governs the size of the spherical cap $B(x,t)$ as defined in \eqref{e.cap}. 
Fourier Analysis on the sphere 
naturally leads to the Gegenbauer polynomial expansion. For any $\lambda \geq 0$ and $g \in L^1([-1,1], (1-x^2)^{\lambda - 1/2})$ we have 
$$ g(x) \sim \sum_{n=0}^{\infty} \widehat{g}(n,\lambda) \frac{n+\lambda}{\lambda} C_n^{\lambda}(x),$$
where $C_n^{\lambda}(x)$ are the Gegenbauer polynomials and $ \widehat{g}(n, \lambda)$ is the associated Fourier coefficient
$$  \widehat{g}(n,\lambda) = \frac{\Gamma(\lambda+1)}{\Gamma(\lambda + 1/2) \Gamma(1/2)} \frac{1}{C_n^{\lambda}(1)} \int_{-1}^{1} g(x) C_n^{\lambda}(x) (1-x^2)^{\lambda - 1/2}dx.$$
Having established the notation, we can now state our main ingredient.

\begin{thm}[Bilyk \& Dai \cite{bilyk}] Let $\lambda = (d-1)/2$. There exists a constant $c_d > 0$ depending only on the dimensions such that for all $f \in L^2_{w_{\lambda}}[-1,1]$  
and all $Z = \left\{z_1, \dots, z_N\right\} \subset \mathbb{S}^d$ we have
$$ {D}^2_{L^2, f}  (Z) \geq c_d \min_{1\leq k \leq c_d N^{1/d}} \left| \widehat{f}(k, \lambda)\right|^2.$$
\end{thm}

Our next ingredient is an identity for Gegenbauer polynomials that we found mentioned in a paper by R. Schneider \cite{schneider}. Since we were unable to find the identity anywhere else, we quickly supply the simple proof. The proof is based on the Rodrigues formula
$$ C_n^{\lambda}(t) = \frac{(-1)^n}{2^n n!} \frac{\Gamma(\lambda + 1/2) \Gamma(n+2\lambda)}{\Gamma(2\lambda) \Gamma(\lambda + n + 1/2)}  \cdot (1-t^2)^{-\lambda+1/2} \frac{d^n}{dt^n} \left[ (1-t^2)^{n + \lambda - 1/2} \right].$$
We will, for simplicity of exposition, abbreviate the numerical constant by $a_n^{\lambda}$. With this abbreviation, we can now evaluate a weighted integral over Gegenbauer polynomials over the interval $[\alpha, 1]$ in closed form.
\begin{lemma} \label{lem:schneider}
Abbreviating the constant $a_n^{\lambda} \in \mathbb{R}$ in the Rodrigues formula
$$C_n^{\lambda}(t) = a_n^{\lambda} ~ (1-t^2)^{-\lambda + 1/2} \frac{d^n}{dt^n} (1-t^2)^{n + \lambda - 1/2},$$
then, for all $-1 < \alpha < 1$,
$$ \int_{\alpha}^1 C_n^{\lambda}(t) (1-t^2)^{\lambda - 1/2} dt = -\frac{ a_n^{\lambda}}{ a_{n-1}^{\lambda + 1}} \cdot (1-\alpha^2)^{\lambda + 1/2} \cdot C_{n-1}^{\lambda + 1}(\alpha).$$
\end{lemma}
\begin{proof} The proof is straightforward: we first simplify the expression on the left-hand side and then differentiate to argue that it leads to the right-hand side. It is easy to see that both sides of the equation coincide when $\alpha = 1$ (since both the left-hand side and the right-hand side vanish in that case) and so the result follows.
We start by arguing that appealing to the Rodrigues formula
\begin{align*}
   \frac{1}{a_{n-1}^{\lambda + 1}} C_{n-1}^{\lambda + 1}(t) &=     \frac{1}{a_{n-1}^{\lambda + 1}} a_{n-1}^{\lambda + 1 } (1-t^2)^{-\lambda - 1/2} \frac{d^{n-1}}{dt^{n-1}} (1-t^2)^{(n-1) + (\lambda +1) - 1/2} \\
   &=  (1-t^2)^{-\lambda - 1/2} \frac{d^{n-1}}{dt^{n-1}} (1-t^2)^{n + \lambda - 1/2}.
\end{align*}
Therefore
\begin{equation} \label{need}
  - a_n^{\lambda} (1-t^2)^{\lambda + 1/2} \frac{1}{a_{n-1}^{\lambda + 1}}  C_{n-1}^{\lambda + 1}(t) = - a_n^{\lambda}  \frac{d^{n-1}}{dt^{n-1}} (1-t^2)^{n + \lambda - 1/2}.
  \end{equation}
We will now use this identity to prove this Lemma. We have already established the validity when $\alpha = 1$. Differentiating the left-hand side
$$ \frac{\partial}{\partial \alpha}  \int_{\alpha}^1 C_n^{\lambda}(t) (1-t^2)^{\lambda - 1/2} dt = - C_n^{\lambda}(\alpha) (1-\alpha^2)^{\lambda - 1/2}.$$
Differentiating the right-hand side in the statement of the Lemma becomes much simpler after using \eqref{need} since
\begin{align*}
\frac{\partial}{\partial \alpha}  \left[ -\frac{ a_n^{\lambda}}{ a_{n-1}^{\lambda + 1}} \cdot (1-\alpha^2)^{\lambda + 1/2} \cdot C_{n-1}^{\lambda + 1}(\alpha) \right] &=
\frac{\partial}{\partial \alpha}  \left[   - a_n^{\lambda}  \frac{d^{n-1}}{d\alpha ^{n-1}} (1-\alpha^2)^{n + \lambda - 1/2} \right] \\
&= - a_n^{\lambda}    \frac{d^{n}}{d\alpha ^{n}} (1-\alpha^2)^{n + \lambda - 1/2} 
\end{align*}
Checking whether both derivatives are the same is exactly equivalent to the Rodrigues formula
$$ C_n^{\lambda}(t) =  a_n^{\lambda} (1-t^2)^{-\lambda + 1/2}   \frac{d^{n}}{dt^{n}} (1-t^2)^{n + \lambda - 1/2}.$$
\end{proof}

\subsection{Proof of Theorem \ref{t.main}}
\begin{proof} We can now combine all the different ingredients. The single radius spherical cap discrepancy corresponds to the choice
$$ f_t(\tau) = \chi_{[t,1]}(\tau).$$
We start by computing its Gegenbauer coefficients 
$$  \widehat{f_t}(n;\lambda) = \frac{\Gamma(\lambda+1)}{\Gamma(\lambda + 1/2) \Gamma(1/2)} \frac{1}{C_n^{\lambda}(1)} \int_{t}^{1} C_n^{\lambda}(x) (1-x^2)^{\lambda - 1/2}dx.$$
The leading term is comprised of a constant depending only on $\lambda$ and $1/C_n^{\lambda}(1)$. It is known, see Abramowitz \& Stegun \cite[Equation 22.14.2]{ab}, that
$$ C_n^{\lambda}(1) = 
\binom{n+2\lambda - 1}{n} \lesssim n^{2\lambda -1}$$
 with a constant again  depending only on $\lambda$.
Using Lemma \ref{lem:schneider}, we deduce 
\begin{align*}
  \left| \widehat{f_t}(n,\lambda) \right|  &\gtrsim_{\lambda} \frac{1}{n^{2\lambda -1}}  \left| \int_{t}^{1} C_n^{\lambda}(x) (1-x^2)^{\lambda - 1/2}dx \right| \\
  &= \frac{1}{n^{2\lambda -1}}  \left| a_n^{\lambda} (a_{n-1}^{\lambda + 1})^{-1} (1-t^2)^{\lambda + 1/2} C_{n-1}^{\lambda + 1}(t) \right|.
  \end{align*}
  Some further simplification is in order. The Rodrigues formula, and thus the coefficients $a_n^{\lambda}$, are known and, ignoring constants that depend only on $\lambda$,
  \begin{align*}
 a_n^{\lambda}  &= \frac{(-1)^n}{2^n n!} \frac{\Gamma(\lambda + 1/2) \Gamma(n+2\lambda)}{\Gamma(2 \lambda) \Gamma(\lambda + n +1/2)} \sim_{\lambda}  \frac{(-1)^n}{2^n n!} n^{\lambda-1/2}
\end{align*}
Therefore,
$$ \left| - a_n^{\lambda} (a_{n-1}^{\lambda + 1})^{-1} (1-t^2)^{\lambda + 1/2}  \right| \sim_{\lambda,t} \frac{1}{n^2}$$
and the lower bound simplifies to
$$   \left| \widehat{f_t}(n,\lambda) \right|  \gtrsim_{\lambda,t} \frac{1}{n^{2\lambda + 1}} \left| C_{n-1}^{\lambda + 1}(t)  \right|.$$
Assuming that $t$ is $(\lambda+1)$-gegenbadly approximable (notice that $\lambda +1 = (d+1)/2$), we arrive at 
$$ \left| C_{n-1}^{\lambda + 1}(t) \right| \geq c_{t}  n^{\lambda}.$$
Putting this all together, we find that   
$$ \left|  \widehat{f_{t}}(n, \lambda) \right| \gtrsim_{t, \lambda}  \frac{1}{n^{\lambda +1}}.$$
Then
$$ 
|\widehat{f_{t}}(n,\lambda)|^2 \gtrsim_{t, \lambda} \frac{1}{n^{2\lambda + 2}} = \frac{1}{n^{d+1}}.$$
Finally, appealing to the Theorem of Bilyk \& Dai, we complete the proof: 
$$ D^2_{L^2, f_t}  (Z) \geq c_d \min_{1\leq k \leq c_d N^{1/d}} |\widehat{f_{t}}(n,\lambda)|^2 \gtrsim \frac{1}{N^{1 + \frac{1}{d}}},$$
\end{proof}

We  conclude with a quick analysis of what this argument gives on $\mathbb{S}^1$. Theorem \ref{thm:2} is not applicable since $d \equiv 1~(\mbox{mod}~4)$, however, Theorem \ref{thm:1} can still be applied. Since $d=1$, we obtain $\lambda = (d-1)/2 = 0$ and $\lambda +1 = 1$, hence   Theorem \ref{thm:1} guarantees existence of suitable $-1 < t< 1$ such that
$$   \left| \widehat{f_t}(n, \lambda) \right|  \gtrsim_{\lambda} \frac{1}{n^{2\lambda + 1}} \left| C_{n-1}^{\lambda + 1}(t)  \right| \gtrsim \frac{1}{n^2}.$$
Then 
$$ D^2_{L^2, f_t}  (Z) \geq c_d \min_{1\leq k \leq c_d N^{1/d}} |\widehat{f_{t}}(n,\lambda)|^2 \gtrsim \frac{1}{N^4},$$
This is far from optimal: it is not too difficult to see that the single spherical cap $L^2-$discrepancy on $\mathbb{S}^1$ can be as small as $\sim 1/N$ but not smaller than that. This rate is attained for equispaced points. 
We conclude with a basic heuristic: if $d = 4k + 1$, then $\lambda = (d-1)/2 = 2k$ and $\lambda+ 1 = 2k+1$ is an odd integer, then  Theorem \ref{thm:1} states that $$  \left| C_n^{\lambda + 1}(x) \right| \leq c_x \cdot  n^{\lambda - 1} \qquad \mbox{for infinitely many}~n.$$
If one were to assume the existence of a converse bound, meaning special values $-1<x<1$ such that $\left| C_n^{\lambda + 1}(x) \right| \gtrsim c_x n^{\lambda -1}$, then this would imply the weaker lower bound
$   | \widehat{f_t}(n;\lambda) |   \gtrsim 1/n^{\lambda+2} $
and 
$$ \mathcal{D}^2_{L^2, f_{\alpha}, N} \gtrsim \frac{1}{N^{1 + \frac{3}{d}}}.$$
We do not know whether this is true; however, if true, it would be strongest result that can possibly be obtained with this method. The first open case is $\mathbb{S}^5$.\\

\textbf{Acknowledgments.}
The authors were supported by the NSF DMS-2054606 (DB, MM) and DMS-2123224 (SS). They are grateful to Giacomo Gigante for spotting an error in a preliminary version of the manuscript.

\end{document}